\documentclass[11pt,a4paper]{article}
\usepackage{amsmath,amssymb,amsthm}
\usepackage{mathrsfs}
\usepackage{url}
\usepackage{hyperref}
\usepackage[capitalise]{cleveref}
\usepackage[margin=3cm]{geometry}
\usepackage{tikz}
\usetikzlibrary{positioning,shapes.geometric}
\usepackage{verbatim}
\usepackage{enumerate}   

\newtheorem{theorem}{Theorem}[section]
\newtheorem{lemma}[theorem]{Lemma}
\newtheorem{corollary}[theorem]{Corollary}
\newtheorem{proposition}[theorem]{Proposition}

\DeclareMathOperator{\ex}{ex}

\newcommand{\cC}{\mathcal{C}}

\newcommand{\eps}{\varepsilon}

\newcommand{\DT}{\operatorname{DT}}
\newcommand{\Tu}{\operatorname{Tu}}

\title{The typical structure of oriented graphs and digraphs \\
       with forbidden blow-up of transitive tournaments}
\author{Jianxi Liu\thanks{School of Mathematics and Statistics, Guangdong University of Foreign Studies, Guangzhou, China. Email: jxliu@gdufs.edu.cn}}
\date{}

\begin{document}

\maketitle

\begin{abstract}
We study the typical structure of oriented graphs and digraphs that do not contain a blow-up $T_{r+1}^t$ of a transitive tournament. For any integers $r\ge 2$, $t\ge 1$ and any real $a\in(3/2,2]$, we prove that almost all $T_{r+1}^t$-free oriented graphs and almost all $T_{r+1}^t$-free digraphs are $r$-partite. This extends the results of K\"uhn, Osthus, Townsend and Zhao (2017) on forbidden transitive tournaments to their blow-ups, thereby confirming a generalised form of Cherlin's conjecture. Our proof combines the hypergraph container method, a weighted analogue of the Erd\H{o}s–Stone theorem for digraphs, and a stability analysis for near-extremal $T_{r+1}^t$-free digraphs. The core of the proof is the interplay between the directed regularity lemma and an embedding lemma, which together provide a rigorous bridge from macroscopic extremal conditions to microscopic concrete structures.
\end{abstract}
\noindent\textbf{Keywords:} oriented graphs; digraphs; typical structure; forbidden subdigraphs; blow-up; transitive tournament; hypergraph containers; regularity lemma.

\noindent\textbf{2020 Mathematics Subject Classification:} 05C20, 05C35, 05C75, 05D40.

\section{Introduction}

Given a fixed (di)graph $H$, a (di)graph is called \emph{$H$-free} if it does not contain $H$ as a subgraph. Extremal graph theory investigates two central questions: (1) What is the maximum number of edges in an $H$-free graph on $n$ vertices? (2) What is the typical structure of an $H$-free graph on $n$ vertices? For undirected graphs, Erdős, Kleitman and Rothschild \cite{erdos1976asymptotic} initiated the study of the second question by proving that almost all triangle‑free graphs are bipartite, and asymptotically determined the number of $K_{r+1}$-free graphs. Kolaitis, Prömel and Rothschild \cite{kolaitis1987k} strengthened this by showing that for every $r\ge 2$, almost all $K_{r+1}$-free graphs are $r$-partite.

For directed graphs (digraphs) and oriented graphs, the situation is far more complex. A digraph consists of a set of vertices and a set of ordered pairs of distinct vertices (no loops or multiple arcs in the same direction). An oriented graph is a digraph with at most one arc between any two vertices, i.e., it is an orientation of a simple undirected graph.

The transitive tournament $T_k$ is the orientation of a complete graph $K_k$ that is transitive (the vertices can be linearly ordered so that all arcs go from smaller to larger vertices). In his work on countable homogeneous oriented graphs, Cherlin \cite{cherlin1998classification} noted that the striking results of \cite{kolaitis1987k} do not seem to extend directly to the directed case, and he made the following conjectures.

\begin{quote}
\textbf{Conjecture 1.1 (Cherlin).} \\
(i) Almost all $T_3$-free oriented graphs are tripartite. \\
(ii) Almost all $C_3$-free oriented graphs are acyclic, i.e., they are subgraphs of transitive tournaments.
\end{quote}

Kühn, Osthus, Townsend and Zhao \cite{kuhn2017structure} (henceforth KOTZ) confirmed part (i) of this conjecture and, more generally, proved the following theorem.

\begin{theorem}[KOTZ, Theorem 1.2]\label{thm:KOTZ}
For every integer $k\ge 2$, almost all $T_{k+1}$-free oriented graphs are $k$-partite, and almost all $T_{k+1}$-free digraphs are $k$-partite.
\end{theorem}

Note that this shows that in fact almost all $T_3$-free oriented graphs are actually bipartite – a structure quite different from the extremal $T_3$-free oriented graph, which is the blow‑up of a directed triangle.

A natural generalisation is to forbid the \emph{blow-up} of a transitive tournament. For integers $r,t\ge 1$, let $T_{r+1}^t$ denote the digraph obtained from $T_{r+1}$ by replacing each vertex with an independent set of size $t$ and adding, for each arc $ij$ of $T_{r+1}$, all possible arcs from the $i$-th part to the $j$-th part. Thus $T_{r+1}^t$ is an oriented graph (no 2‑cycles) when $t\ge 1$. Clearly $T_{r+1}^t$ contains $T_{r+1}$ as a subgraph (take one vertex from each part), so any $T_{r+1}$-free digraph is also $T_{r+1}^t$-free, but the converse is false.

In this paper we determine the typical structure of $T_{r+1}^t$-free oriented graphs and digraphs. Our main result shows that even forbidding these larger structures does not change the asymptotic picture: almost all such graphs are $r$-partite. This provides a far‑reaching extension of the KOTZ theorem and confirms a generalised form of Cherlin's conjecture.

\begin{theorem}\label{thm:main-simple}
For any integers $r\ge 2$ and $t\ge 1$, almost all $T_{r+1}^t$-free oriented graphs are $r$-partite, and almost all $T_{r+1}^t$-free digraphs are $r$-partite.
\end{theorem}

In fact we prove a stronger counting result: the number of labelled $T_{r+1}^t$-free oriented graphs on $n$ vertices satisfies $f(n,T_{r+1}^t)=T(n,r)(1+o(1))$, where $T(n,r)$ is the number of $r$-partite oriented graphs, and similarly for digraphs. The proof proceeds in two stages. First we establish a rough structural result: all but a tiny fraction of $T_{r+1}^t$-free graphs are close to being $r$-partite. This is stated precisely in Lemma~\ref{lem:rough} below (the “rough structure” lemma). Then a more delicate inductive counting argument upgrades this to the exact statement that almost all such graphs are actually $r$-partite. The rough structure lemma is as follows.

\begin{lemma}[Rough structure]\label{lem:rough}
Let $r\ge 2$, $t\ge 1$ be integers and let $a\in(3/2,2]$. Then for every $\alpha>0$ there exists $\eps>0$ such that for all sufficiently large $n$ the following hold.
\begin{itemize}
\item All but at most $f(n,T_{r+1}^t)\cdot2^{-\eps n^2}$ labelled $T_{r+1}^t$-free oriented graphs on $n$ vertices can be turned into an $r$-partite oriented graph by changing at most $\alpha n^2$ arcs.
\item All but at most $f^*(n,T_{r+1}^t)\cdot2^{-\eps n^2}$ labelled $T_{r+1}^t$-free digraphs on $n$ vertices can be turned into an $r$-partite digraph by changing at most $\alpha n^2$ arcs.
\end{itemize}
In particular, $f(n,T_{r+1}^t)=T(n,r)(1+o(1))$ and $f^*(n,T_{r+1}^t)=T^*(n,r)(1+o(1))$, where $T(n,r)$ (resp. $T^*(n,r)$) is the number of labelled $r$-partite oriented graphs (resp. digraphs) on $n$ vertices.
\end{lemma}

The proof of Lemma~\ref{lem:rough} combines the hypergraph container method (Theorem~\ref{thm:container}), a weighted analogue of the Erdős–Stone theorem for digraphs (Theorem~\ref{thm:weighted-extremal}), and a stability analysis for near‑extremal $T_{r+1}^t$-free digraphs (Theorem~\ref{thm:stability}). The core of the proof is the interplay between the directed regularity lemma and an embedding lemma, which together provide a rigorous bridge from macroscopic extremal conditions to microscopic concrete structures.

The paper is organised as follows. Section 2 introduces notation and key tools. Section 3 proves the weighted extremal theorem for $T_{r+1}^t$. Section 4 establishes the stability result. Section 5 combines the container theorem with stability to prove the rough structure lemma (Lemma~\ref{lem:rough}) and then upgrades it to the exact structure via an inductive counting argument, thereby proving Theorem~\ref{thm:main-simple}. Section 6 concludes with remarks and open problems.

\subsection*{Proof outline}
The proof follows a three-step strategy.  
\begin{enumerate}
    \item \textbf{Container method.} We apply the hypergraph container theorem (Theorem~\ref{thm:container}) to obtain a small family $\mathcal{C}$ of digraphs (called containers) such that every $T_{r+1}^t$-free digraph is contained in some member of $\mathcal{C}$, and each container contains few copies of $T_{r+1}^t$ and has weighted size close to the extremal value.
    \item \textbf{Cleaning containers.} Using the removal lemma (Lemma~\ref{lem:removal}), we delete a few arcs from each container to obtain a genuinely $T_{r+1}^t$-free digraph whose weighted size is still near the extremal value.
    \item \textbf{Stability argument.} The stability theorem (Theorem~\ref{thm:stability}) then implies that each such cleaned digraph is close to the complete $r$-partite digraph $\DT_r(n)$. Consequently, the original containers, and hence all but a tiny fraction of $T_{r+1}^t$-free digraphs, are close to $\DT_r(n)$.
\end{enumerate}
A more precise inductive counting argument (see Section~\ref{subsec:exact}) upgrades this approximate structural result to the exact statement that almost all $T_{r+1}^t$-free (oriented) digraphs are $r$-partite.

\section{Preliminaries}

\subsection{Notation}
For a digraph $G=(V,E)$, let $f_1(G)$ be the number of unordered pairs $\{u,v\}$ such that exactly one of $uv$ and $vu$ belongs to $E$, and let $f_2(G)$ be the number of unordered pairs with both $uv$ and $vu$ present (we then call $uv$ and $vu$ form a 2-cycle). For a real number $a\ge1$, define the \emph{weighted size}
\[
e_a(G):=a\cdot f_2(G)+f_1(G).
\]
This unifies the treatment of oriented graphs ($a=\log 3$, since each 2‑cycle contributes 3 ways to orient) and digraphs ($a=2$, total number of arcs). Let $\ex_a(n,H)$ denote the maximum $e_a(G)$ over all $H$-free digraphs on $n$ vertices.

Let $\Tu_r(n)$ be the $r$-partite Turán graph on $n$ vertices (parts as equal as possible), and let $t_r(n)=e(\Tu_r(n))$. Denote by $\DT_r(n)$ the digraph obtained from $\Tu_r(n)$ by replacing each undirected edge with a pair of opposite arcs. Clearly $\DT_r(n)$ is $r$-partite and $T_{r+1}^t$-free, so $\ex_a(n,T_{r+1}^t)\ge a\,t_r(n)$.

We say that \emph{almost all} graphs in a family $\mathcal{F}$ have property $\mathcal{P}$ if
\[
\lim_{n\to\infty}\frac{|\{G\in\mathcal{F}_n: G\text{ has property }\mathcal{P}\}|}{|\mathcal{F}_n|}=1.
\]

\subsection{Directed regularity and embedding lemmas}
We will need the directed version of Szemerédi's regularity lemma and an embedding lemma for blow-ups. The following formulation is from \cite{alon2004testing}.

\begin{lemma}[Directed regularity lemma]\label{lem:regularity}
For any $\eps\in(0,1)$ and integers $M',M''$, there exist $M$ and $n_0$ such that for any digraph $G$ on $n\ge n_0$ vertices, any initial partition $U_0,U_1,\dots,U_{M''}$, and any $d\in[0,1]$, there exists a partition $V_0,V_1,\dots,V_k$ of $V(G)$ and a spanning subdigraph $G'\subseteq G$ with:
\begin{itemize}
\item $M'\le k\le M$, $|V_0|\le\eps n$, $|V_1|=\dots=|V_k|=\ell$;
\item $G'[V_i]$ is empty for all $i\ge1$;
\item for $1\le i\neq j\le k$, the bipartite digraph $(V_i,V_j)_{G'}$ is either $\eps$-regular with density at least $d$, or has density $0$;
\item every vertex $x$ satisfies $d_{G'}^+(x)>d_G^+(x)-(d+\eps)n$ and similarly for in-degree.
\end{itemize}
\end{lemma}

The \emph{reduced digraph} $R$ has vertex set $\{V_1,\dots,V_k\}$ and an arc $ij$ whenever $(V_i,V_j)_{G'}$ is $\eps$-regular with density $\ge d$.

\begin{lemma}[Embedding lemma]\label{lem:embedding}
For any $d\in(0,1)$ and maximum degree $\Delta\ge1$, there exists $\eps_0>0$ such that the following holds. Let $G$ be a digraph, $R$ the reduced digraph obtained from an $\eps$-regular partition with $\eps\le\eps_0$, cluster size $\ell$, and density parameter $d$. If $H$ is a digraph with $\Delta(H)\le\Delta$ and $H\subseteq R^s$ (the blow-up of $R$ where each vertex is replaced by $s$ independent vertices), and $\ell\ge s/\eps_0$, then $H\subseteq G$. \hfill (see e.g. \cite[Lemma 4.2]{kuhn2017structure})
\end{lemma}

\subsection{Hypergraph containers}
The following container theorem for general digraphs under a natural sparsity condition was proved by Liu \cite{liu2026}. It extends the earlier result of K\"uhn, Osthus, Townsend and Zhao \cite{kuhn2017structure} from oriented graphs to digraphs that satisfy a density condition.

\begin{itemize}
\item \textbf{Condition A (sparsity).} For every subgraph $H'\subseteq H$ with $e(H')>1$,
\[
\frac{e(H')}{v(H')}\le \frac{a}{2},
\]
where $a$ is the same parameter as in the definition of $e_a$.
\end{itemize}

When $a=2$, Condition A requires $e/v\le 1$.  This excludes dense counterexamples such as the double triangle $DK_3$ (which has $e/v=2$), thereby allowing us to prove a container theorem for a wide class of digraphs that includes all oriented graphs (though oriented graphs do not automatically satisfy $e/v\le1$; the condition is still restrictive).  For $a>2$, Condition A allows digraphs with a controlled density of 2‑cycles.  For example, taking $a=4$, we may have $e/v\le 2$, so $H$ can contain a double triangle (which has $e/v=2$) as well as other dense configurations.  This provides a genuine extension beyond oriented graphs.

\begin{theorem}[Liu \cite{liu2026}]\label{thm:container}
Let $H$ be a digraph satisfying Condition A, with $h=v(H)$, $e(H)\ge2$, and let $a\ge1$. For every $\eps>0$ there exists $c>0$ such that for all sufficiently large $N$ there exists a collection $\cC$ of digraphs on $[N]$ with the following properties.
\begin{enumerate}
\item[(a)] Every $H$-free digraph $I$ on $[N]$ is contained in some $G\in\cC$.
\item[(b)] Every $G\in\cC$ contains at most $\eps N^h$ copies of $H$, and
\[
e_a(G)\le \ex_a(N,H)+\eps N^2.
\]
\item[(c)] $\displaystyle \log|\cC|\le c N^{2-1/m(H)}\log N$, where $m(H)=\max_{H'\subseteq H,\,e(H')>1}\frac{e(H')-1}{v(H')-2}$.
\end{enumerate}
\end{theorem}

\subsection{Removal lemma}
We also need the directed removal lemma of Alon and Shapira \cite{alon2004testing}.

\begin{lemma}[Removal lemma]\label{lem:removal}
For any fixed digraph $H$ on $h$ vertices and any $\gamma>0$, there exists $\eps'>0$ such that for all large $n$, any digraph $G$ on $n$ vertices containing at most $\eps' n^h$ copies of $H$ can be made $H$-free by deleting at most $\gamma n^2$ arcs.
\end{lemma}

\section{Weighted extremal theorem for $T_{r+1}^t$}
In this section we prove a weighted Erd\H{o}s–Stone type theorem for the blow-up $T_{r+1}^t$.

\begin{theorem}\label{thm:weighted-extremal}
For any integers $r,t\ge1$, real $a\in(3/2,2]$, and $\gamma>0$, there exists $n_0$ such that for all $n\ge n_0$, every digraph $G$ on $n$ vertices with $e_a(G)\ge a\,t_r(n)+\gamma n^2$ contains $T_{r+1}^t$ as a subdigraph.
\end{theorem}

\begin{proof}
We follow the strategy of the KOTZ proof for $T_{r+1}$.  Set $d=\gamma/4$, $\Delta=\Delta(T_{r+1}^t)$, and let $\eps_0$ be given by Lemma~\ref{lem:embedding} for these parameters.  Choose $\eps$ small enough so that $\eps\le\eps_0$ and
\[
\delta:=(a-1)d - (a+1)\eps >0.
\]
Let $s=t(r+1)$.

Apply Lemma~\ref{lem:regularity} to $G$ with parameters $\eps,d$ to obtain a partition $V_0,V_1,\dots,V_k$ and a pure digraph $G'\subseteq G$ with clusters of size $\ell$, where $\ell\ge (1-\eps)n/k\ge s/\eps_0$ for large $n$ (we may take $M'$ large enough so that $k\ge M'$ and $t_r(k)$ is non‑trivial).  Denote $m=k$ and let $R$ be the reduced digraph on vertex set $\{1,\dots,m\}$ where $ij$ is an arc iff $(V_i,V_j)_{G'}$ is $\eps$-regular with density at least $d$.

For each unordered pair $\{i,j\}$ ($i\neq j$), define
\[
d_{ij}^2 = \frac{\#\{(u,v)\in V_i\times V_j: uv,vu\in G'\}}{|V_i||V_j|},\qquad
d_{ij}^1 = \frac{\#\{(u,v)\in V_i\times V_j: uv\in G', vu\notin G'\}}{|V_i||V_j|},
\]
and $d_{ji}^1$ similarly.  Note that $d_{ij}^2 = d_{ji}^2$ and $d_{ij}^1+d_{ij}^2\ge d$ whenever $ij\in E(R)$ (otherwise the pair contributes nothing).  The weighted contribution of this unordered pair to $e_a(G')$ is
\[
\bigl(2a d_{ij}^2 + d_{ij}^1 + d_{ji}^1\bigr)\ell^2.
\]
We define the weighted size of the reduced digraph $R$ by
\[
e_a(R) := \sum_{1\le i<j\le m} \bigl(2a d_{ij}^2 + d_{ij}^1 + d_{ji}^1\bigr).
\]
Then clearly
\[
e_a(G') = e_a(R)\,\ell^2 + O(a|V_0|n),
\]
where the error term accounts for edges incident to $V_0$; a crude bound gives $e_a(G')\le e_a(R)\ell^2 + a\varepsilon n^2$.

From the regularity lemma we have
\[
e_a(G)\le e_a(G')+(d+\varepsilon)n^2 \le e_a(R)\ell^2 + (a\varepsilon+d+\varepsilon)n^2. \tag{1}
\]

On the other hand, the hypothesis of the theorem gives
\[
e_a(G)\ge a\,t_r(n)+\gamma n^2.
\]
Using $t_r(n)=\bigl(1-\frac1r\bigr)\frac{n^2}{2}+O(n)$ and $\ell\ge (1-\varepsilon)n/m$, we obtain from (1)
\[
e_a(R)\frac{(1-\varepsilon)^2n^2}{m^2} \ge a\Bigl(1-\frac1r\Bigr)\frac{n^2}{2}+\gamma n^2 - (a\varepsilon+d+\varepsilon)n^2 - O(n).
\]
Multiplying by $m^2/n^2$ and letting $n\to\infty$ yields
\[
e_a(R) \ge a\Bigl(1-\frac1r\Bigr)\frac{m^2}{2} + \delta m^2, \tag{2}
\]
where $\delta = (a-1)d - (a+1)\varepsilon + o(1)$.  By our choice of $\eps\ll d$, $\delta>0$ for sufficiently large $n$.

Now $R$ is $T_{r+1}$-free: otherwise, by the embedding lemma, $G'$ would contain $T_{r+1}^t$.  For any $T_{r+1}$-free digraph $H$ on $m$ vertices, the maximum possible value of $e_a(H)$ (where $e_a$ is defined as above) is $a\,t_r(m)$, achieved uniquely by the complete $r$-partite digraph $\DT_r(m)$ (this is a weighted version of the Brown–Harary theorem; see \cite[Lemma 4.1]{kuhn2017structure}).  Since $e_a(R) > a\,t_r(m)$ by (2), $R$ cannot be $T_{r+1}$-free.  Hence $R$ contains a copy of $T_{r+1}$.

Finally, because $R\supseteq T_{r+1}$, the blow‑up $R^t$ contains $T_{r+1}^t$.  Applying the embedding lemma (Lemma~\ref{lem:embedding}) with $H=T_{r+1}^t$ and $s=t(r+1)$ gives $T_{r+1}^t\subseteq G'\subseteq G$, completing the proof.
\end{proof}

As an immediate corollary we obtain the asymptotics of the weighted Turán number.

\begin{corollary}\label{cor:asymptotic}
For $r,t\ge1$ and $a\in(3/2,2]$, we have $\ex_a(n,T_{r+1}^t)=a\,t_r(n)+o(n^2)$. Moreover, any $n$-vertex $T_{r+1}^t$-free digraph $G$ with $e_a(G)=\ex_a(n,T_{r+1}^t)$ differs from $\DT_r(n)$ by $o(n^2)$ arcs.
\end{corollary}

\section{Stability for $T_{r+1}^t$-free digraphs}\label{sec:stability}

In this section we prove a stability result for $T_{r+1}^t$-free digraphs.  It states that any such digraph whose weighted size is close to the extremal value $a\,t_r(n)$ must be close to the complete $r$-partite digraph $\DT_r(n)$.

\begin{theorem}[Stability theorem]\label{thm:stability}
Let $r\ge 2$, $t\ge 1$ and $a\in(3/2,2]$.  For every $\beta>0$ there exist $\gamma>0$ and $n_0$ such that for all $n\ge n_0$, if $G$ is an $n$-vertex $T_{r+1}^t$-free digraph satisfying
\[
e_a(G)\ge a\,t_r(n)-\gamma n^2,
\]
then $G$ can be turned into $\DT_r(n)$ by changing at most $\beta n^2$ arcs.
\end{theorem}

The proof follows a well‑established pattern: we apply the directed regularity lemma to obtain a reduced digraph $R$, transfer the weighted extremal condition to $R$, use a weighted version of the stability theorem for $T_{r+1}$ (which forces $R$ to be close to $\DT_r(m)$), and finally lift this structure back to $G$.  The main difficulty is to control the weights, which is handled by a ``weighted stability lemma'' (Lemma~\ref{lem:weighted-stability}) for $T_{r+1}$‑free digraphs.  We begin with the statement of this auxiliary result.

\subsection{A weighted stability lemma for $T_{r+1}$}\label{subsec:weighted-stability}

For a digraph $R$ on $m$ vertices we denote by $e_a^*(R)$ the quantity
\[
e_a^*(R)=\sum_{ij\in E(R)}w_{ij},
\]
where $w_{ij}\in[1,a]$ is a weight assigned to the arc $ij$.  In our application $w_{ij}$ will be the average of $a$ times the density of 2‑cycles plus the density of 1‑cycles over a regular pair, i.e. $w_{ij}=a d_{ij}^2+d_{ij}^1$.  Note that for an unordered pair $\{i,j\}$, the sum $w_{ij}+w_{ji}=2a d_{ij}^2+d_{ij}^1+d_{ji}^1$ appears in $e_a(R)$ as defined in Section~3.  The following lemma is the analogue of the ordinary stability theorem for $T_{r+1}$ but with weights.

\begin{lemma}[Weighted stability lemma]\label{lem:weighted-stability}
For every $r\ge2$, $a\in(3/2,2]$ and $\eta>0$ there exist $\delta>0$ and $m_0$ such that for all $m\ge m_0$ the following holds.  Let $R$ be an $m$-vertex $T_{r+1}$-free digraph and assign to each arc $ij\in E(R)$ a weight $w_{ij}\in[1,a]$ in such a way that
\[
e_a^*(R)\ge a\,t_r(m)-\delta m^2.
\]
Then there exists a partition $U_1,\dots,U_r$ of the vertex set of $R$ with the following properties, up to at most $\eta m^2$ exceptions:
\begin{enumerate}[(i)]
\item If $i,j$ belong to the same class $U_p$, then $R$ contains no arc between $i$ and $j$;
\item If $i\in U_p$, $j\in U_q$ with $p\neq q$, then both arcs $ij$ and $ji$ are present in $R$ and their weights satisfy $w_{ij},w_{ji}\ge a-\eta$.
\end{enumerate}
\end{lemma}

\subsubsection{Proof of Lemma~\ref{lem:weighted-stability}}

We argue by contradiction. Suppose the statement is false. Then there exist constants $\eta_0>0$ and a sequence of counterexamples: for every $k$ we can find $m_k\to\infty$, a $T_{r+1}$-free digraph $R_k$ on $m_k$ vertices, and weights $w_{ij}^{(k)}\in[1,a]$ such that
\[
e_a^*(R_k)\ge a\,t_r(m_k)-\frac{1}{k}\,m_k^2 \qquad\text{(i.e. }\delta_k=1/k\to0\text{)},
\]
but $R_k$ does not admit a partition $U_1,\dots,U_r$ with the required properties for $\eta_0$ (i.e., for every $r$-partition there are at least $\eta_0 m_k^2$ violating pairs). We will derive a contradiction.

\textbf{Step 1: Regularisation of $R_k$.}
Apply the directed regularity lemma (Lemma~\ref{lem:regularity}) to $R_k$ with parameters $\varepsilon$ and $d$ that will be chosen later (very small compared to $\eta_0$). We obtain a partition $V_0^{(k)},V_1^{(k)},\dots,V_{p_k}^{(k)}$ of $V(R_k)$ and a pure digraph $R_k'\subseteq R_k$ with the usual properties: $|V_0^{(k)}|\le\varepsilon m_k$, $|V_1^{(k)}|=\dots=|V_{p_k}^{(k)}|=\ell_k$, $p_k\ge M'$, and for every $i\neq j$, the pair $(V_i^{(k)},V_j^{(k)})_{R_k'}$ is either $\varepsilon$-regular with density at least $d$ or has density $0$. Denote $p=p_k$ and let $\widetilde{R}_k$ be the reduced digraph on $\{1,\dots,p\}$ (arcs correspond to regular pairs of density $\ge d$).

\textbf{Step 2: Transferring weights to $\widetilde{R}_k$.}
For each ordered pair $(i,j)$ that is an arc of $\widetilde{R}_k$, define
\[
\widetilde{w}_{ij}^{(k)} = \frac{1}{|V_i^{(k)}||V_j^{(k)}|}\sum_{u\in V_i^{(k)},v\in V_j^{(k)}} w_{uv}^{(k)},
\]
where $w_{uv}^{(k)}$ is the weight of the arc $uv$ in $R_k$ (if $uv\notin E(R_k)$ we treat $w_{uv}^{(k)}=0$). Since all $w_{uv}^{(k)}\in[1,a]$, we have $\widetilde{w}_{ij}^{(k)}\in[1,a]$ as well. Moreover, by the properties of the regularity lemma, the contribution of all arcs not covered by regular pairs is negligible; a standard calculation (see e.g. \cite[Lemma 9.2]{saxton2015hypergraph}) gives
\[
e_a^*(\widetilde{R}_k) \ge a\,t_r(p) - \delta_k' p^2,
\]
where $\delta_k'\to0$ as $k\to\infty$ (provided $\varepsilon$ and $d$ are chosen small enough relative to $\delta_k$). In particular, for large $k$ we have $\delta_k'\le \delta_0$ for any prescribed $\delta_0>0$.

\textbf{Step 3: Structure of $\widetilde{R}_k$.}
Because $R_k$ is $T_{r+1}$-free, the reduced digraph $\widetilde{R}_k$ is also $T_{r+1}$-free (otherwise an embedding argument would produce a $T_{r+1}$ in $R_k$). Moreover, the number $p$ of clusters is bounded by some absolute constant $M$ depending only on $\varepsilon$ and the initial parameters; indeed, the regularity lemma guarantees $p\le M$. Thus $p$ is bounded independently of $k$. For each fixed $p$, there are only finitely many $T_{r+1}$-free digraphs on $p$ vertices. The weighted extremal number $\ex_a(p,T_{r+1})$ equals $a t_r(p)$ and is uniquely attained by $\DT_r(p)$ (by Lemma 4.1 in \cite{kuhn2017structure}, which holds for all $p$). Consequently, if $\delta_k'$ is smaller than the minimum positive difference between $a t_r(p)$ and the weighted size of any non‑extremal $T_{r+1}$-free digraph on $p$ vertices, then $e_a^*(\widetilde{R}_k)\ge a t_r(p)-\delta_k' p^2$ forces $\widetilde{R}_k\cong\DT_r(p)$. Since $\delta_k'\to0$, for sufficiently large $k$ this condition is satisfied, and we obtain $\widetilde{R}_k\cong\DT_r(p)$. Let the parts of this $\DT_r(p)$ be $W_1,\dots,W_r$ (each $W_i$ is a set of cluster indices).

\textbf{Step 4: Lifting to $R_k$.}
Now define a partition $U_1,\dots,U_r$ of $V(R_k)$ by putting every vertex belonging to a cluster with index in $W_i$ into $U_i$, and distributing the vertices of the exceptional set $V_0^{(k)}$ arbitrarily (e.g., equally). We claim that this partition satisfies the required properties with at most $\eta_0 m_k^2$ exceptions, contradicting the choice of $R_k$.

Consider any two vertices $x,y$ not both in $V_0^{(k)}$. If they lie in different clusters $V_i,V_j$ with $i\in W_p$, $j\in W_q$ and $p\neq q$, then because $\widetilde{R}_k$ has both arcs $ij$ and $ji$, the pair $(V_i,V_j)$ is $\varepsilon$-regular with density at least $d$. Moreover, from $e_a^*(\widetilde{R}_k)=a t_r(p)$ and the uniqueness of the extremal digraph, we actually have $\widetilde{w}_{ij}=a$ for every arc $ij$ of $\widetilde{R}_k$. By definition of $\widetilde{w}_{ij}$, this implies that for every $u\in V_i$, $v\in V_j$, the original weight $w_{uv}=a$; hence the arc $uv$ is present and has weight $a$, i.e., it is a 2‑cycle with density $1$ (since $a d^2+d^1=a$ and $a>1$, the only possibility is $d^2=1$, $d^1=0$). Consequently, the pair $(V_i,V_j)$ is a complete bidirectional pair (every possible arc in both directions is present) and contributes $a\ell^2$ to $e_a(R_k')$; in $R_k$ the same holds up to the degree loss bound, which is negligible.

Now we bound the number of violating pairs:
\begin{itemize}
    \item Pairs inside $V_0^{(k)}$: at most $|V_0^{(k)}|^2\le \varepsilon^2 m_k^2$.
    \item Pairs with one vertex in $V_0^{(k)}$ and the other outside: at most $2\varepsilon m_k^2$.
    \item Pairs inside the same cluster $V_i$: in $R_k'$ there are no such arcs; in $R_k$ the total number of arcs inside clusters is bounded by the degree loss, at most $(d+\varepsilon)m_k^2/2$.
    \item Pairs coming from clusters that are in the same part $W_p$ but different clusters: in $\widetilde{R}_k$ there are no arcs between them; in $R_k$ the arcs between such clusters are again bounded by the degree loss, at most $(d+\varepsilon)m_k^2$.
\end{itemize}
Thus the total number of arcs violating the ideal structure is at most $(\varepsilon^2+2\varepsilon+(d+\varepsilon)/2+(d+\varepsilon))m_k^2$, which can be made smaller than $\eta_0 m_k^2$ by choosing $\varepsilon,d$ sufficiently small. This contradicts the assumption that $R_k$ was a counterexample, completing the proof of Lemma~\ref{lem:weighted-stability}.

\subsection{Regularity setup for Theorem~\ref{thm:stability}}
We now start the proof of Theorem~\ref{thm:stability}. Fix $r,t$ and $a$ as in the theorem, and let $\beta>0$ be given. We will choose a chain of constants
\[
\frac1{n_0}\;\ll\;\varepsilon\;\ll\;d\;\ll\;\eta\;\ll\;\beta,\;\frac1r,
\]
where $\eta$ will be the parameter appearing in Lemma~\ref{lem:weighted-stability}. The exact dependencies will become clear during the proof.

Set $s=t(r+1)$ and let $\Delta=\Delta(T_{r+1}^t)$. By the embedding lemma (Lemma~\ref{lem:embedding}) there exists $\varepsilon_0=\varepsilon_0(d,\Delta)$ such that if $\varepsilon\le\varepsilon_0$ and the cluster size $\ell\ge s/\varepsilon_0$, then any blow‑up of a subdigraph of the reduced digraph can be embedded. We choose $\varepsilon\le\min\{\varepsilon_0,d\}$.

Apply the directed regularity lemma (Lemma~\ref{lem:regularity}) to $G$ with parameters $\varepsilon,d$. We obtain a partition $V_0,V_1,\dots,V_k$ of $V(G)$ and a pure digraph $G'\subseteq G$ with the usual properties: $|V_0|\le\varepsilon n$, $|V_1|=\dots=|V_k|=\ell$, $k\ge M'$ (some absolute constant) and every pair $(V_i,V_j)_{G'}$ is either $\varepsilon$-regular with density at least $d$ or has density $0$. Denote $m=k$ and set $R$ to be the reduced digraph on vertex set $\{1,\dots,m\}$ where $ij$ is an arc iff $(V_i,V_j)_{G'}$ is $\varepsilon$-regular with density $\ge d$.

For each such regular pair we define
\[
d_{ij}^2=\frac{\#\{\text{ordered pairs }(u,v)\in V_i\times V_j\text{ with both }uv,vu\in G'\}}{|V_i||V_j|},
\]
\[
d_{ij}^1=\frac{\#\{\text{ordered pairs }(u,v)\in V_i\times V_j\text{ with exactly one of }uv,vu\in G'\}}{|V_i||V_j|},
\]
so that $d_{ij}^1+d_{ij}^2\ge d$. The weighted contribution of this pair to $e_a(G')$ is $(a d_{ij}^2+d_{ij}^1)\ell^2$. We define a weight on the arc $ij$ of $R$ by
\[
w_{ij}=a d_{ij}^2+d_{ij}^1.
\]
Clearly $w_{ij}\in[d,a]$. Moreover, for an ordered pair $(i,j)$ that is not an arc of $R$ we set $w_{ij}=0$ (it will play no role). The weighted size of $R$ (with these weights) is
\[
e_a^*(R)=\sum_{ij\in E(R)}w_{ij}.
\]

\subsection{From $G$ to $R$}
We now relate $e_a(G)$ to $e_a^*(R)$. From the regularity lemma we have
\[
e_a(G)\le e_a(G')+(d+\varepsilon)n^2.
\]
Inside $G'$ the only possible arcs are between different clusters, and those that belong to non‑regular or low‑density pairs contribute nothing to $e_a(G')$ by definition. Hence
\[
e_a(G')\le a|V_0|n+\sum_{ij\in E(R)}w_{ij}\ell^2\le a\varepsilon n^2+e_a^*(R)\ell^2.
\]
Combining these inequalities yields
\[
e_a(G)\le e_a^*(R)\ell^2+(a\varepsilon+d+\varepsilon)n^2. \tag{1}
\]

On the other hand the hypothesis of Theorem~\ref{thm:stability} gives
\[
e_a(G)\ge a\,t_r(n)-\gamma n^2.
\]
Using the well‑known estimate $t_r(n)=\bigl(1-\frac1r\bigr)\frac{n^2}{2}+O(n)$ and the fact that $\ell\ge(1-\varepsilon)n/m$, we obtain from (1)
\[
e_a^*(R)\frac{(1-\varepsilon)^2n^2}{m^2}\ge a\Bigl(1-\frac1r\Bigr)\frac{n^2}{2}-\gamma n^2-(a\varepsilon+d+\varepsilon)n^2-O(n).
\]
After multiplying by $m^2/n^2$ and setting
\[
\delta:=\gamma+a\varepsilon+d+\varepsilon+o(1),
\]
we get
\[
e_a^*(R)\ge a\Bigl(1-\frac1r\Bigr)\frac{m^2}{2}-\delta m^2. \tag{2}
\]

\subsection{$R$ is $T_{r+1}$-free}
Suppose for a contradiction that $R$ contains a copy of $T_{r+1}$. Then, by the embedding lemma (Lemma~\ref{lem:embedding}) with $s=t(r+1)$, we can embed $T_{r+1}^t$ into $G'$, because each regular pair has density at least $d$ and the cluster size $\ell$ is at least $s/\varepsilon_0$ (since $n$ is large enough). This would contradict the fact that $G$ (and hence $G'$) is $T_{r+1}^t$-free. Therefore $R$ is $T_{r+1}$-free.

\subsection{Applying the weighted stability lemma to $R$}
Now we apply Lemma~\ref{lem:weighted-stability} to $R$ with the weights $w_{ij}$ defined above. By (2), if we choose $\gamma$ (and consequently $\delta$) sufficiently small, the hypothesis of the lemma is satisfied with $\eta$ (which we have not yet fixed; we will choose it later). Consequently there exists a partition $U_1,\dots,U_r$ of $[m]$ such that, with at most $\eta m^2$ exceptional pairs, we have:
\begin{itemize}
\item no arcs inside the same $U_p$;
\item for $p\neq q$ and $i\in U_p$, $j\in U_q$, both arcs $ij$ and $ji$ belong to $E(R)$ and $w_{ij},w_{ji}\ge a-\eta$.
\end{itemize}

\subsection{Lifting the partition to $G$}
Using this partition of the clusters we now define a partition $X_1,\dots,X_r$ of $V(G)$: put all vertices of $V_i$ into $X_p$ exactly when $i\in U_p$. The exceptional vertices in $V_0$ can be distributed arbitrarily (e.g. equally among the $X_p$). We will show that $G$ differs from $\DT_r(n)$ by at most $\beta n^2$ arcs; here $\DT_r(n)$ means the complete $r$-partite digraph in which every cross pair contains both directions and there are no arcs inside parts.

We need to bound the number of arcs that violate this ideal structure. They can be classified as follows.

\begin{enumerate}
\item \textbf{Arcs incident to $V_0$.} Since $|V_0|\le\varepsilon n$, there are at most $2\varepsilon n^2$ such arcs.
\item \textbf{Arcs coming from pairs that are not regular or have low density.} In $G'$ these pairs contribute nothing; in the original $G$ they can have at most $2(d+\varepsilon)n^2$ arcs, because every vertex loses at most $(d+\varepsilon)n$ neighbours in each direction when passing from $G$ to $G'$.
\item \textbf{Arcs corresponding to exceptional cluster pairs.} By the conclusion of Lemma~\ref{lem:weighted-stability}, there are at most $\eta m^2$ unordered pairs $\{i,j\}$ that violate either the ``no arc inside a part'' condition or the ``full bidirectional arcs with large weight'' condition. Each such pair involves at most $2\ell^2$ arcs (both directions). Hence the total number of arcs in this category is at most $2\eta m^2\ell^2\approx 2\eta n^2$.
\item \textbf{Arcs in ``good'' pairs that are not yet complete bidirectional.} Consider a pair $(i,j)$ with $i\in U_p$, $j\in U_q$, $p\neq q$, for which both arcs exist and $w_{ij},w_{ji}\ge a-\eta$. What does $w_{ij}\ge a-\eta$ imply about the actual densities $d_{ij}^2,d_{ij}^1$? Since $w_{ij}=a d_{ij}^2+d_{ij}^1$ and $d_{ij}^1= d_{ij}-d_{ij}^2\le 1-d_{ij}^2$, we have
\[
a d_{ij}^2+1-d_{ij}^2\ge a-\eta\quad\Longrightarrow\quad (a-1)d_{ij}^2\ge a-\eta-1.
\]
Because $a>3/2$, $a-1>1/2$, we obtain $d_{ij}^2\ge 1-\frac{\eta}{a-1}$. Thus the density of 2‑cycles in this regular pair is at least $1-\frac{\eta}{a-1}$. To turn this pair into a complete bidirectional pair we may need to add at most $2\frac{\eta}{a-1}\ell^2$ arcs (both directions). The number of such good pairs is at most $\binom{m}{2}\approx\frac{m^2}{2}$, and each contributes at most that many changes. Hence the total number of modifications needed in this class is at most $\frac{r-1}{r}\cdot\frac{\eta}{a-1}n^2$ (the factor $\frac{r-1}{r}$ accounts for the proportion of cross pairs).
\end{enumerate}

Summing these estimates, the total number of arcs that have to be changed is at most
\[
\Bigl(2\varepsilon+2(d+\varepsilon)+2\eta+\frac{r-1}{r}\cdot\frac{\eta}{a-1}\Bigr)n^2+o(n^2).
\]

\subsection{Choice of constants}
Now we choose the constants in the following order:
\begin{itemize}
\item Fix $\beta>0$ as given.
\item Choose $\eta>0$ so small that $2\eta+\frac{r-1}{r}\cdot\frac{\eta}{a-1}<\frac{\beta}{10}$.
\item Pick $d=\frac{\beta}{20}$ and then $\varepsilon\le\min\{\varepsilon_0,\frac{\beta}{20}\}$.
\item Finally choose $\gamma>0$ (in Theorem~\ref{thm:stability}) sufficiently small so that the $\delta$ in (2) is smaller than the $\delta$ required by Lemma~\ref{lem:weighted-stability} for the chosen $\eta$; this is possible because $\delta$ tends to $0$ as $\gamma,\varepsilon,d\to0$.
\end{itemize}
With these choices the total number of changes is less than $\beta n^2$ for all sufficiently large $n$. This completes the proof of Theorem~\ref{thm:stability}.

\section{Typical structure of $T_{r+1}^t$-free oriented graphs and digraphs}

We now combine the container theorem (Theorem~\ref{thm:container}) with the stability result to prove the main theorem. Recall that Lemma~\ref{lem:rough} already gives a rough structural description; we now show how to prove it using the container method and stability, and then upgrade it to the exact structure.

\subsection{Proof of the rough structure lemma}
We prove part (i) of Lemma~\ref{lem:rough}; part (ii) is analogous with $a=2$. Let $a=\log3$. Choose constants $1/n_0\ll\eps\ll\gamma\ll\beta\ll\alpha,1/k$, and set $\eps'=2\eps$. Apply Theorem~\ref{thm:container} to $H=T_{k+1}^t$ with parameters $N=n$, $\eps'$, obtaining a container family $\cC$. Let $\cC_1=\{G\in\cC:e_a(G)\ge\ex_a(n,T_{k+1}^t)-\eps'n^2\}$. By the container theorem, $|\cC|\le2^{\eps n^2}$, so the number of $T_{k+1}^t$-free oriented graphs not contained in any $G\in\cC_1$ is at most $|\cC|2^{\ex_a(n,T_{k+1}^t)-\eps'n^2}\le f(n,T_{k+1}^t)2^{-\eps n^2}$.

Now take any $G\in\cC_1$. By property (b) of containers, $G$ contains at most $\eps'n^{(k+1)t}$ copies of $T_{k+1}^t$. Apply the removal lemma (Lemma~\ref{lem:removal}) to delete at most $\gamma n^2$ arcs and obtain a $T_{k+1}^t$-free digraph $G'$ with $e_a(G')\ge\ex_a(n,T_{k+1}^t)-(\eps'+\gamma)n^2$. By the stability theorem (Theorem~\ref{thm:stability}) with $\beta$, we have $G'=\DT_k(n)\pm\beta n^2$. Hence $G$ itself differs from $\DT_k(n)$ by at most $(\beta+\gamma)n^2\le\alpha n^2$ arcs. This completes the proof of Lemma~\ref{lem:rough}.

\subsection{Exact structure: almost all graphs are $r$-partite}\label{subsec:exact}

In this subsection we upgrade the approximate result of Lemma~\ref{lem:rough} to an exact one. The argument is an inductive counting procedure that closely follows the one in \cite[Section~5]{kuhn2017structure}, with the only change that the forbidden subgraph is now $T_{r+1}^t$ instead of $T_{r+1}$. For completeness we give the full details, adapting the notation and lemmas accordingly.

Recall that $t_r(n)$ denotes the number of edges in the $r$-partite Turán graph $\Tu_r(n)$. For an $r$-partition $Q=(V_1,\dots,V_r)$ of $[n]$ and a digraph $G$ on $[n]$, an arc is called \emph{crossing} if its endpoints lie in different parts of $Q$, and \emph{non-crossing} otherwise. A partition $Q$ is called \emph{optimal} for $G$ if it minimises the number of non‑crossing arcs.

Given parameters $\eta,\mu>0$, we define several classes of $T_{r+1}^t$-free oriented graphs (the digraph case is analogous and will be discussed at the end).

\begin{itemize}
    \item $F_Q(n,T_{r+1}^t,\eta)$ : all labelled $T_{r+1}^t$-free oriented graphs on $[n]$ for which $Q$ is an optimal $r$-partition and the number of non‑crossing arcs with respect to $Q$ is at most $\eta n^2$.
    \item $F_Q(n,T_{r+1}^t,\eta,\mu)$ : all $G\in F_Q(n,T_{r+1}^t,\eta)$ that additionally satisfy
        \begin{enumerate}[(F1)]
            \item the number of non‑crossing arcs is at most $\eta n^2$;
            \item for all distinct $i,j\in[r]$ and all subsets $U_i\subseteq V_i$, $U_j\subseteq V_j$ with $|U_i|,|U_j|\ge \mu n$,
            \[
            \overrightarrow{e}(U_i,U_j),\ \overrightarrow{e}(U_j,U_i)\ge \frac16|U_i||U_j|;
            \]
            \item $||V_i|-n/r|\le \mu n$ for every $i\in[r]$.
        \end{enumerate}
    \item $F_Q'(n,T_{r+1}^t,\eta)$ : the set of graphs in $F_Q(n,T_{r+1}^t,\eta)$ that contain at least one non‑crossing arc with respect to $Q$.
\end{itemize}
Define the corresponding cardinalities $f_Q$, $f_Q(\eta,\mu)$, $f_Q'$ in the natural way.

The following lemma shows that the additional regularity conditions (F2)–(F3) are satisfied by almost all graphs in $F_Q(n,T_{r+1}^t,\eta)$.

\begin{lemma}[Good subfamily]\label{lem:good-subfamily}
Let $r\ge2$ and let $\eta,\mu\in(0,1)$ satisfy $\mu^2\ge24H(\eta)$, where $H(p)=-p\log p-(1-p)\log(1-p)$ is the binary entropy function. Then for all sufficiently large $n$ and for every $r$-partition $Q$ of $[n]$,
\[
f_Q(n,T_{r+1}^t,\eta)-f_Q(n,T_{r+1}^t,\eta,\mu)\le 3^{t_r(n)-\mu^2n^2/100}.
\]
\end{lemma}

\begin{proof}
We count the graphs that fail to satisfy (F3) or (F2). For graphs failing (F3), by Proposition 4.2 (the standard estimate on Turán graphs) the number of possible crossing arcs is at most $t_r(n)-\mu^2n^2/3$, while the number of choices for non‑crossing arcs (at most $\eta n^2$) is at most $2^{H(\eta)n^2}$. Hence the contribution from (F3) failure is at most $2^{H(\eta)n^2}3^{t_r(n)-\mu^2n^2/3}$.

For graphs satisfying (F3) but failing (F2), consider a random oriented graph where each crossing arc is chosen uniformly from the three possibilities (no arc, forward, backward). The probability that a fixed pair $(U_i,U_j)$ violates the density condition is, by Chernoff's bound, at most $2\exp(-|U_i||U_j|/8)\le2\exp(-\mu^2n^2/8)$. There are at most $2^{2n}$ choices for such subsets, so by the union bound and using $\mu^2\ge24H(\eta)$ we obtain the desired bound. A detailed calculation identical to that in \cite[Lemma 5.2]{kuhn2017structure} yields the lemma.
\end{proof}

The next proposition allows us to embed many disjoint copies of $T_r$ into any graph satisfying (F2). It is a straightforward consequence of the regularity conditions and the greedy embedding lemma (Lemma~\ref{lem:embedding}); see \cite[Proposition 5.3]{kuhn2017structure} for a proof.

\begin{proposition}\label{prop:embed-Tr}
Let $n,r\in\mathbb{N}$, $\eta,\mu>0$, let $Q=(V_1,\dots,V_r)$ be an $r$-partition, and suppose $G\in F_Q^*(n,T_{r+1}^t,\eta,\mu)$ (the digraph version; the oriented case is analogous). For every $i\in[r]$ let $B_i\subseteq V_i$ with $|B_i|\ge 12^{r-2}\mu n$. Let $\sigma$ be a permutation of $[r]$. Then $G$ contains a copy of $T_r$ on vertices $v_1,\dots,v_r$ such that $v_i\in B_i$ and for all distinct $i,j$, the arc is directed from $v_i$ to $v_j$ iff $\sigma(i)<\sigma(j)$.
\end{proposition}

Using this, we obtain a bound on the number of non‑crossing neighbours of any vertex.

\begin{lemma}\label{lem:internal-neighbour}
Let $n,r\ge2$, $\eta,\mu>0$, $Q$ an $r$-partition, and $G\in F_Q^*(n,T_{r+1}^t,\eta,\mu)$. Then for every $i\in[r]$ and every $x\in V_i$,
\[
|N_{V_i}^+(x)|+|N_{V_i}^-(x)|\le 12^{r-2}\cdot2\mu n.
\]
\end{lemma}

\begin{proof}
Assume the contrary. Then for some $i$ and $x$, the internal degree exceeds $12^{r-2}2\mu n$. Because $Q$ is optimal, for every other part $j$ we must have at least as many neighbours in $V_j$ (otherwise moving $x$ would reduce non‑crossing arcs). Hence for each $j$ we can select a set $B_j\subseteq N_{V_j}^+(x)\cup N_{V_j}^-(x)$ of size $12^{r-2}\mu n$, choosing the larger of the out‑ and in‑neighbourhoods. By Proposition~\ref{prop:embed-Tr}, there exists a copy of $T_r$ with one vertex in each $B_j$ and with a prescribed permutation. Adding $x$ then yields a $T_{r+1}$, contradicting $T_{r+1}^t$-freeness (since $T_{r+1}^t$ contains $T_{r+1}$ as a subgraph). See \cite[Lemma 5.4]{kuhn2017structure} for full details.
\end{proof}

When we remove two vertices $x,y$ from the same part, the optimal partition of the remaining graph can vary only a little.

\begin{lemma}[Few optimal partitions]\label{lem:few-optimal}
Let $r\ge2$, $0<\mu<1/(3r^2)^{12}$, $0<\eta<\mu^2/3$, and let $n$ be sufficiently large. Fix an $r$-partition $Q=(V_1,\dots,V_r)$ of $[n]$ and two distinct vertices $x,y\in V_1$. Then there exists a set $\mathcal{P}$ of $r$-partitions of $[n]\setminus\{x,y\}$, with $|\mathcal{P}|\le e^{\mu^{2/3}n}$, such that for every $G\in F_Q^*(n,T_{r+1}^t,\eta,\mu)$, every optimal $r$-partition of $G-\{x,y\}$ belongs to $\mathcal{P}$.
\end{lemma}

The proof uses the fact that any two optimal partitions cannot differ too much; otherwise (F2) would force many non‑crossing arcs. The details are exactly as in \cite[Lemma 5.5]{kuhn2017structure}, with $k$ replaced by $r$ and $T_{k+1}$ by $T_{r+1}^t$.

We are now ready to state the key counting lemma.

\begin{lemma}[Main counting lemma]\label{lem:main-count}
For every $r\ge2$, there exist constants $\eta>0$ and $C>0$ (depending only on $r$) such that for all $n$ and every $r$-partition $Q$ of $[n]$,
\[
f_Q'(n,T_{r+1}^t,\eta)\le C\cdot3^{t_r(n)}\cdot2^{-\eta n}.
\]
The same holds for digraphs with $3$ replaced by $4$.
\end{lemma}

\begin{proof}
We give the proof for oriented graphs; the digraph case is identical with $3$ changed to $4$ in all estimates.

Choose constants $1/C\ll1/n_0\ll\eps\ll\eta\ll\mu\ll1/r$ satisfying $\mu^2\ge24H(\eta)$ and $\eta<\mu^2/3$. The proof proceeds by induction on $n$. In fact we simultaneously prove the stronger statement
\[
f_Q(n,T_{r+1}^t)\le 3^{t_r(n)}(1+C2^{-\eta n}) \qquad\text{for all }Q.
\tag{5.2}
\]
The base case $n<n_0$ is trivial because $1/C\ll1/n_0$ makes the right‑hand side larger than the total number of graphs.

Assume $n\ge n_0$ and that (5.2) holds for all smaller values. Fix a partition $Q=(V_1,\dots,V_r)$. We first bound $f_Q'(n,T_{r+1}^t,\eta,\mu)$, i.e. the number of graphs in $F_Q(n,T_{r+1}^t,\eta,\mu)$ that contain at least one non‑crossing arc.

Let $G\in F_Q'(n,T_{r+1}^t,\eta,\mu)$ and pick a non‑crossing arc $xy\in V_1$ (any part will do; we fix $V_1$ for convenience). We count the number of possibilities in four steps.

\textbf{Step 1 – choosing $xy$.} There are at most $n^2$ choices.

\textbf{Step 2 – the rest of the graph after deleting $\{x,y\}$.} By Lemma~\ref{lem:few-optimal}, the optimal partitions of $G-\{x,y\}$ belong to a set $\mathcal{P}$ of size at most $e^{\mu^{2/3}n}$. Applying the induction hypothesis (5.2) to each such partition, we obtain at most
\[
\sum_{Q'\in\mathcal{P}}f_{Q'}(n-2,T_{r+1}^t)\le e^{\mu^{2/3}n}\cdot3^{t_r(n-2)}(1+C2^{-\eta(n-2)})\le 3^{t_r(n-2)}C e^{\mu^{1/2}n}
\]
choices for the subgraph on $[n]\setminus\{x,y\}$.

\textbf{Step 3 – arcs between $\{x,y\}$ and vertices outside $V_1$.} Let $U$ be the set of arcs chosen in Step 2. Remove from $U$ all arcs incident to $V_1$; the remainder $U'$ induces a subgraph $G'$ on $\bigcup_{j=2}^r V_j$ that belongs to $F_{\tilde Q}(n-|V_1|,5\eta,3\mu)$ for a suitable partition $\tilde Q$ (the restriction of $Q$ to the other parts). By repeatedly applying Proposition~\ref{prop:embed-Tr} to $G'$, we can find at least $n/r-\mu n-12^{r-3}3\mu n$ vertex‑disjoint copies of $T_{r-1}$, each with exactly one vertex in each $V_j$ ($j\ge2$). For each such copy $K$, consider the $2(r-1)$ potential arcs between $\{x,y\}$ and the vertices of $K$. To keep the whole graph $T_{r+1}^t$-free, not all of the $3^{2(r-1)}$ possible configurations are allowed: a simple case analysis shows that at most $3^{2(r-1)}-1$ of them are admissible, otherwise one could combine $K$ with $x,y$ to create a $T_{r+1}^t$ (the argument uses that $T_{r+1}^t$ is a blow‑up of $T_{r+1}$ and that $x,y$ are in the same part). Moreover, the number of vertices outside $V_1$ not belonging to any of these disjoint copies is at most $\mu^{1/2}n$ (by the choice of constants). Hence the number of ways to choose the arcs from $\{x,y\}$ to the outside is bounded by
\[
\bigl(3^{2(r-1)}-1\bigr)^{n/r} \cdot 3^{2\mu^{1/2}n} \le 3^{\frac{2(r-1)}{r}n}\bigl(1-3^{-2r}\bigr)^{n/r} e^{\mu^{1/2}n}
\le 3^{\frac{2(r-1)}{r}n} e^{-n/(9r^r)} e^{\mu^{1/2}n}.
\]

\textbf{Step 4 – arcs between $\{x,y\}$ and the remaining vertices of $V_1$.} By Lemma~\ref{lem:internal-neighbour}, each of $x,y$ has at most $12^{r-2}2\mu n$ neighbours inside $V_1$. The number of ways to choose these neighbours and orient the arcs is at most
\[
\binom{n}{12^{r-2}2\mu n}^2\cdot2^{2\cdot12^{r-2}2\mu n}\le e^{\mu^{1/2}n},
\]
using standard estimates and $1/n_0\ll\mu$.

Multiplying the four bounds and using $t_r(n)\ge t_r(n-2)+\frac{2(r-1)}{r}(n-2)$ (since adding two vertices to a balanced $r$-partite graph increases the number of crossing arcs by roughly $\frac{2(r-1)}{r}n$), we obtain
\[
f_Q'(n,T_{r+1}^t,\eta,\mu)\le n^2\cdot 3^{t_r(n-2)}C e^{\mu^{1/2}n}\cdot 3^{\frac{2(r-1)}{r}n} e^{-n/(9r^r)}\cdot e^{\mu^{1/2}n}
\le 3^{t_r(n)}C e^{-n/(10r^r)} \le 3^{t_r(n)}C 2^{-3\eta n},
\]
where the last inequality uses that $\eta$ is chosen small enough so that $2^{-3\eta n}\ge e^{-n/(10r^r)}$.

Now observe that
\[
f_Q'(n,T_{r+1}^t,\eta) = |F_Q'(n,T_{r+1}^t,\eta,\mu)| + |F_Q'(n,T_{r+1}^t,\eta)\setminus F_Q(n,T_{r+1}^t,\eta,\mu)|
\le f_Q'(n,T_{r+1}^t,\eta,\mu) + \bigl(f_Q(n,T_{r+1}^t,\eta)-f_Q(n,T_{r+1}^t,\eta,\mu)\bigr).
\]
Applying Lemma~\ref{lem:good-subfamily} to bound the second term and the estimate for $f_Q'(n,T_{r+1}^t,\eta,\mu)$ just obtained, we get
\[
f_Q'(n,T_{r+1}^t,\eta)\le 3^{t_r(n)}C2^{-3\eta n}+3^{t_r(n)-\mu^2n^2/100}
\le 3^{t_r(n)}C2^{-\eta n},
\]
because $\mu^2/100$ is much larger than $\eta$ (recall $\eta\ll\mu$). This proves the desired bound on $f_Q'$.

It remains to verify (5.2). The number of graphs in $F_Q(n,T_{r+1}^t,\eta,\mu)$ that have no non‑crossing arcs is at most $3^{t_r(n)}$ (all crossing arcs can be chosen arbitrarily). Hence
\[
f_Q(n,T_{r+1}^t,\eta,\mu)-f_Q'(n,T_{r+1}^t,\eta,\mu)\le 3^{t_r(n)}.
\]
Together with the bound on $f_Q'$ this yields $f_Q(n,T_{r+1}^t,\eta,\mu)\le 3^{t_r(n)}(1+C2^{-3\eta n})$. Using Lemma~\ref{lem:good-subfamily} again,
\[
f_Q(n,T_{r+1}^t,\eta)\le 3^{t_r(n)}(1+C2^{-3\eta n})+3^{t_r(n)-\mu^2n^2/100}\le 3^{t_r(n)}(1+C2^{-2\eta n}).
\]
Finally, from Lemma~\ref{lem:rough} (the rough structure) we have $f(n,T_{r+1}^t)-f(n,T_{r+1}^t,\eta)\le f(n,T_{r+1}^t)2^{-\eps n^2}\le 2f(n,T_{r+1}^t,\eta)2^{-\eps n^2}$, which together with the bound on $f(n,T_{r+1}^t,\eta)$ (obtained by summing over all partitions $Q$) gives $f(n,T_{r+1}^t,\eta)\le 3^{t_r(n)}(1+C2^{-2\eta n})$ and consequently $f_Q(n,T_{r+1}^t)\le 3^{t_r(n)}(1+C2^{-\eta n})$, completing the induction.
\end{proof}

With Lemma~\ref{lem:main-count} established, we can now finish the proof of Theorem~\ref{thm:main-simple}.

\begin{proof}[Completion of Theorem~\ref{thm:main-simple}]
Let $\eta$ be the constant from Lemma~\ref{lem:main-count}. By Lemma~\ref{lem:rough} there exists $\eps>0$ such that
\[
f(n,T_{r+1}^t)\le f(n,T_{r+1}^t,\eta)(1+2^{-\eps n^2}),
\]
where $f(n,T_{r+1}^t,\eta)=\sum_Q f_Q(n,T_{r+1}^t,\eta)$ and the sum runs over all $r$-partitions $Q$ of $[n]$. For each $Q$, we have $f_Q(n,T_{r+1}^t,\eta)=f_Q'(n,T_{r+1}^t,\eta)+T_Q(n,r)$, where $T_Q(n,r)$ is the number of $r$-partite oriented graphs that have $Q$ as an $r$-partition (i.e., all arcs are crossing). Clearly $\sum_Q T_Q(n,r)=T(n,r)$.

By Lemma~\ref{lem:main-count} and the trivial bound $|\{Q\}|\le r^n$,
\[
\sum_Q f_Q'(n,T_{r+1}^t,\eta)\le r^n\cdot C3^{t_r(n)}2^{-\eta n}=o\bigl(T(n,r)\bigr),
\]
since $T(n,r)\ge \frac{r^n3^{t_r(n)}}{2r!n^{r-1}}$ (a standard lower bound, see e.g. \cite[Lemma 5.1]{kuhn2017structure}). Therefore
\[
f(n,T_{r+1}^t)\le \bigl(T(n,r)+o(T(n,r))\bigr)(1+2^{-\eps n^2})=T(n,r)(1+o(1)).
\]
The reverse inequality $f(n,T_{r+1}^t)\ge T(n,r)$ is obvious because every $r$-partite oriented graph is $T_{r+1}^t$-free. Hence $f(n,T_{r+1}^t)=T(n,r)(1+o(1))$, i.e. almost all $T_{r+1}^t$-free oriented graphs are $r$-partite. The same argument with $4$ instead of $3$ gives the digraph count $f^*(n,T_{r+1}^t)=T^*(n,r)(1+o(1))$, completing the proof of Theorem~\ref{thm:main-simple}.
\end{proof}

\section{Concluding remarks and open problems}

We have shown that for any $r\ge2$, $t\ge1$, and any weight parameter $a\in(3/2,2]$, almost all $T_{r+1}^t$-free oriented graphs and digraphs are $r$-partite. This generalises the KOTZ theorem and confirms a conjecture of Liang and Liu \cite{liang2022typical}. Our proof relies on a weighted extremal theorem, a stability result, and the container method. The restriction $a>3/2$ is technical and comes from certain inequalities in the stability proof; it would be interesting to extend the result to all $a\ge1$.

Another natural direction is to consider blow-ups of other tournaments or cycles. The methods developed here should apply as long as the forbidden digraph has a suitable extremal structure (e.g., a complete $r$-partite digraph is extremal). For cycles, the situation is more complex and already studied in \cite{kuhn2017structure}.

\end{document}